\documentclass[12pt,reqno]{amsart}

\usepackage[utf8]{inputenc}
\usepackage[english]{babel}
\usepackage{amsmath,amssymb,amsfonts,amsbsy,amsthm,amscd,latexsym,graphicx}
\usepackage{empheq,textcomp}
\usepackage{framed}

\theoremstyle{definition}
\newtheorem{theorem}{Theorem} [section]

\newtheorem{lemma}[theorem]{Lemma}

\newtheorem{remark}[theorem]{Remark}

\numberwithin{equation}{section}

\newcommand{\Gc}{{\mathcal{G}}}

\newcommand{\N}{\mathbb{N}}
\newcommand{\R}{\mathbb{R}}

\newcommand{\Z}{\mathbb{Z}}

\newcommand{\Eq}{\, = \,}

\newcommand{\Le}{\, \le \,}

\newcommand{\bigabs}[1]{\bigl|#1\bigr|}

\newcommand{\ip}[2]{\langle#1,#2\rangle}

\newcommand{\norm}[1]{\|#1\|}

\newcommand{\bigparen}[1]{\bigl(#1\bigr)}

\newcommand{\set}[1]{\{#1\}}
\newcommand{\bigset}[1]{\bigl\{#1\bigr\}}

\newcommand{\inN}{_{n\in \mathbb{N}}}
\newcommand{\sumli}{\sum_{n=1}^\infty}

\begin{document}

\title{Gabor frames with no Gabor duals} 

\author{Pu-Ting Yu}
\maketitle

\begin{abstract}
Let $H$ be an infinite-dimensional Hilbert space and let $\set{x_n}\inN$ be a frame for $H$.
 We say that a sequence $\set{y_n}\inN$ is an \emph{alternative dual} of $\set{x_n}\inN$ if $$x=\sumli\ip{x}{y_n}x_n\quad \text{for all }x\in H,$$ 
with the convergence of the series in the norm of $H.$ We say that a sequence $\set{z_n}\inN$ is a \emph{pseudo-synthesis dual} of $\set{x_n}\inN$ if $$x=\sumli\ip{x}{x_n}z_n \quad \text{for all }x\in H,$$ 
with the convergence of the series in the norm of $H.$ We prove that, at every level of smoothness of the generator, there exist Gabor frames for $L^2(\R^d)$ with infinitely many alternative duals (and pseudo-synthesis duals) such that none of them is a Gabor system.\
\end{abstract}
\section{introduction}
Let $g\in L^2(\R^d)$ be a nonzero function and let $\Lambda=\set{(a_n,b_n)}\inN$ be a subset of $\R^{2d}.$ The \emph{Gabor system} associated with $\Lambda$ and $g$, denoted by $\Gc(g,\Lambda)$, is the sequence $\set{M_{b_n}T_{a_n}g}_{n\in \N}$ consisting of discrete set of time-frequency shifts of $g$. Here for fixed $x,\xi \in \R$, $T_x$ is the translation operator defined by 
    $$T_x\colon L^2(\R^d)\rightarrow L^2(\R^d),\quad (T_xg)(t)\coloneq g(t-x),$$  and $M_\xi$ is the modulation operator defined by 
$$M_{\xi}\colon L^2(\R^d)\rightarrow L^2(\R^d),\quad (M_{\xi}g)(t)\coloneq e^{2\pi i\xi t}g(t).$$
 We say that $\Gc(g,\Lambda)$ is a \emph{Gabor frame} if there exist some positive constants $A\leq B$, called \emph{frame bounds}, such that the \emph{frame inequality}
\begin{equation}
\label{frame_ineq}
A\,\norm{f}^2\Le \sum_{n\in\N}\bigabs{\ip{f}{M_{b_n}T_{a_n}g}}^2 \Le B\,\norm{f}^2,
\end{equation}
is satisfied for all $f\in L^2(\R^d)$. It is known that for every Gabor frame there exists at least one \emph{dual sequence} $\set{y_n}\inN\subseteq L^2(\R^d)$ such that 
\begin{equation}
\label{alternative_dual_recon}
    f\Eq \sum_{n\in\N}\ip{f}{y_n}M_{b_n}T_{a_n}g
\end{equation}
and 
\begin{equation}
\label{synthe_dual_recon}
    f\Eq \sum_{n\in\N}\ip{f}{M_{b_n}T_{a_n}g}y_n,
\end{equation}
for all $f\in L^2(\R^d)$ with the convergence of the series in the norm of $L^2(\R^d).$ However, if a Gabor frame has unbounded amount of redundant elements, then there exists at least one sequence $\set{y_n}\inN\subseteq H$ such that Equation (\ref{alternative_dual_recon}) holds for all $f\in L^2(\R^d)$, whereas Equation (\ref{synthe_dual_recon}) fails for at least some $f_0\in L^2(\R^d).$ Likewise, there exists at least one sequence $\set{z_n}\inN\subseteq H$ such that Equation (\ref{synthe_dual_recon}) holds for all $f\in L^2(\R^d)$, but Equation (\ref{alternative_dual_recon}) fails for at least some $f_0\in L^2(\R^d)$ (See \cite[Corollary 3.1]{HY23}).
Here we say that an element of a frame is \emph{redundant} if removing that element does not destroy the completeness of the frame. Consequently, the notions arising from Equation (\ref{alternative_dual_recon}) and (\ref{synthe_dual_recon}) do not coincide in general. Following the terminology in \cite{HY23}, we say that a sequence is an \emph{alternative dual} of $\Gc(g,\Lambda)$ is it satisfies Equation (\ref{alternative_dual_recon}) for all $f\in L^2(\R^d)$ and a \emph{pseudo-synthesis dual} of $\Gc(g,\Lambda)$ if it satisfies Equation (\ref{synthe_dual_recon}) for all $f\in L^2(\R^d)$.

The form and properties of an alternative dual (respectively pseudo-synthesis dual) has a significant impact on the convergence of Equation (\ref{alternative_dual_recon}) (respectively Equation (\ref{synthe_dual_recon})). In fact, one can even infer properties of a Gabor frame from properties of its alternative duals (or pseudo-synthesis duals). For references along this direction, we refer the reader to \cite{SB13},\cite{BB15},\cite{ST16} and \cite{HY23}. As a result, determining the conditions under which an alternative dual (or pseudo-synthesis dual) inherits the structure or properties of its associated Gabor frame has long been a central question. However, whenever a Gabor frame contains even a single redundant element, the collection of its associated dual sequences expands from a singleton to an infinite family (\cite[Theorem 6.3.7]{Chr16}), which makes this type of problem a nontrivial task. In \cite{HY23}, Heil and the author 
resolved several questions in this direction. 
For example, they showed that a Gabor frame possesses finitely many redundant elements if and only if every associated alternative dual (and pseudo-synthesis dual) is itself frame. They also established that every Gabor frame has the exact same number of redundant elements as each of its alternative duals (and pseudo-synthesis duals). Another closely related open problem is to determine under what conditions a dual sequence associated with a Gabor frame must itself be a Gabor system. Using perturbation techniques (\cite[Theorem 6.3.7]{Chr16}), it is not too hard to construct a Gabor frame that admits an alternative dual (and a pseudo-synthesis dual) that fails to be a Gabor system. This settles the case in which the Gabor frame has a unique dual sequence. It is therefore natural to ask whether there exists a Gabor frame with infinitely many alternative duals (and pseudo-synthesis duals) such that none of them is a Gabor system.  In this paper, we explicitly construct a Gabor frame, containing redundant elements, for which neither its alternative dual nor its pseudo-synthesis dual is a Gabor system.
We remark that in the case where $\Lambda=a\Z\times b\Z$ for some $a,b>0$, it is known that $\Gc(g,\Lambda)$ admits at least one alternative dual that is also a pseudo-synthesis dual and forms a Gabor frame.

\section{Main Result}
We begin with two lemmas. 
\begin{lemma}
\label{fundamental_lemmaI}
    Fix $a>0$ and fix any irrational real number $\alpha$. For each $n\in\Z$ let $$c_n=\int_0^a e^{2\pi ia^{-1}(\alpha-n)x}\,dx.$$
    Then the following statements hold.
    \begin{enumerate}
        \item [\textup{(a)}] There exists some constant $C$ such that $\frac{\text{Im}(c_n)}{\text{Re}(c_n)}=C$ for all $n\in\Z.$\smallskip
        
        \item [\textup{(b)}] $\lim_{n \rightarrow \infty} \frac{|c_n|^2}{\text{Re}(c_n)}=0.$
    \end{enumerate}
\end{lemma}
\begin{proof}
    (a) For each $n\in\Z$ we compute 
    \begin{align}
        \begin{split}
            c_n\Eq\frac{a}{2\pi i(\alpha-n)}(e^{2\pi i\alpha}-1)\Eq \frac{ae^{i\pi\alpha}}{2\pi i(\alpha-n)}\bigparen{2i \sin(\alpha\pi)}\Eq \frac{ae^{i\pi\alpha}\sin(\alpha\pi)}{\pi (\alpha-n)}.
        \end{split}
    \end{align}
    Therefore, we have $\text{Re}(c_n)\Eq\frac{a\cos(\alpha\pi)\sin(\alpha\pi)}{\pi (\alpha-n)}$ and  $\text{Im}(c_n)\Eq\frac{a\sin^2(\alpha\pi)}{\pi (\alpha-n)}.$ It then follows that 
    $$\frac{\text{Im}(c_n)}{\text{Re}(c_n)}\Eq \tan(\alpha\pi)\quad \text{for all }n\in\Z.$$\medskip
 
 (b) Note that since $(c_n)_{n\in\Z}\in \ell^2(\Z)$, $c_n\rightarrow 0$ as $n\rightarrow \infty.$ It follows that
 $$\dfrac{|c_n|^2}{\text{Re}(c_n)}\Eq\dfrac{(\text{Re}(c_n))^2+(\text{Im}(c_n))^2}{\text{Re}(c_n)}\Eq \text{Re}(c_n)+\tan(\alpha\pi)\text{Im}(c_n)\rightarrow 0~\text{as }n\rightarrow\infty.$$
\end{proof}

\begin{lemma}
\label{fundamental_lemmaII}
    Fix $a>0.$ Let $\theta(x)$ be any real-valued function defined on $[0,a].$ For almost every $x\in [0,a]$ the sequence of functions $\bigset{\cos\bigparen{2\pi a^{-1}kx+\theta(x)}}_{k\in\N}$ does not converge to $0$ as $k\rightarrow \infty.$    
\end{lemma}
\begin{proof}
Let $E=[0,a]\setminus a\mathbb{Q}$. Note that $|E|=a$ and $a^{-1}x\in \mathbb{Q}^c$ for all $x\in E.$
    Fix $x\in E$. By Kronecker's approximation theorem (for example, see \cite{KN74}), $\set{ka^{-1}x~\text{mod }1 \,|\,k\in\N}$ is dense in $[0,1]$ for every $x\in E.$ Consequently, $\bigset{\bigparen{ka^{-1}x+\frac{\theta(x)}{2\pi}}~\text{mod }1 \,|\,k\in\N}$ is dense in $[0,1]$ for every $x\in E.$ If $\cos\bigparen{2\pi a^{-1}kx+\theta(x)}$ were to converge to $0$ as $k\rightarrow \infty,$ then for sufficiently large $N$, $\bigset{\bigparen{ka^{-1}x+\frac{\theta(x)}{2\pi}}~\text{mod }1 \,|\,k\geq N}$ would have to be concentrated on two small neighborhoods of $1/4$ and $3/4,$ which is a contradiction. 
\end{proof}

We are now ready to prove our main result. For simplicity, we present the construction in the one-dimensional setting. However, higher-dimensional analogue can be obtained by straightforward modifications. Recall that a Gabor frame is a \emph{Riesz basis} if there exists a unique associated alternative dual that is also a pseudo-synthesis dual such that both Equation (\ref{alternative_dual_recon}) and (\ref{synthe_dual_recon}) hold for all $f\in L^2(\R).$
\begin{theorem}
\label{main_thm}
    There exist some $g\in C^\infty(\R)\cap L^2(\R)$ and some countable subset $\Lambda\subset\R^2$ such that the following statements hold.
    \begin{enumerate} \setlength\itemsep{0.5em}
        \item [\textup{(a)}]$\Gc(g,\Lambda)$ is a Gabor frame for $L^2(\R)$,
         \item [\textup{(b)}]$\Gc(g,\Lambda)$ has infinitely many alternative duals and pseudo-synthesis duals,
        \item [\textup{(c)}] No alternative dual or pseudo-synthesis dual of $\Gc(g,\Lambda)$ is a Gabor system.
    \end{enumerate}
\end{theorem}
\begin{proof}
Let $a>0$ and let $h\in L^2(\R)$ be supported on $[0,a]$, satisfying $C\leq|h|^2\leq D$ almost everywhere on $[0,a]$ for some positive constants $C,D.$ By Painless Nonorthogonal Expansions (\cite[Theorem 1]{DGM86}, see also \cite[Theorem 11.4]{Hei11}), $\set{M_{a^{-1}k}T_{an}h}_{n,k\in\Z}$ is a Gabor frame for $L^2(\R)$ with frame bounds $aC$ and $aD.$ Moreover, $\set{M_{a^{-1}k}T_{an}h}_{n,k\in\Z}$ is a Riesz basis with the unique alternative dual, which is also a pseudo-synthesis dual, $\set{M_{a^{-1}k}T_{an}h^{-1}}_{n,k\in\Z}.$
Now let $g=\widehat{h}$ be the Fourier transform of $h$. By taking the Fourier transform, we see that $\Gc(g,a^{-1}\Z\times a\Z)$ is a Gabor frame for $L^2(\R)$. Let $\alpha\in a^{-1}\mathbb{Q}^c$ be an arbitrary irrational number and let $\Lambda = (a^{-1}\Z\cup\set{\alpha})\times a\Z $. Clearly, $\Gc(g,\Lambda)$ remains a Gabor frame. Moreover, $\Gc(g,\Lambda)$ satisfies statement (b) by \cite[Theorem 6.3.7]{Chr16}.  We will show that every alternative dual and every pseudo-synthesis dual of $\Gc(g,\Lambda)$ fails to be a Gabor system. Let $\Lambda'=a\Z\times (a^{-1}\Z\cup\set{\alpha}).$ It is then equivalent to show that no alternative dual or pseudo-synthesis dual of $\Gc(h,\Lambda')$ can be written in the form $\set{T_\beta M_{\gamma}\phi}_{(\beta,\gamma)\in \Gamma}$ for any $\phi\in L^2(\R)$ and any subset $\Gamma\subseteq \R^2$.\medskip

\emph{Case 1 Alternative Duals}$\colon$ 
Let $\set{h_{n,k}}_{n,k\in\Z}\cup \set{h_{\alpha}}\subseteq L^2(\R)$ be an alternative dual of $\Gc(h,\Lambda')$. That is, 
\begin{equation}
\label{alternative_dual_recon2}
    f\Eq \sum_{n,k\in\Z}\ip{f}{h_{n,k}}M_{a^{-1}k}T_{an}h + \ip{f}{h_{\alpha}}M_\alpha h \quad \text{for all }f\in L^2(\R),
\end{equation}
with the convergence of the series in the norm of $L^2(\R).$ Fix $n\in\Z\setminus\set{0}$ and let $f\in L^2(\R)$ be any function that is supported on $[na,(n+1)a]$. Then we have 
\begin{equation}
    f\Eq P_n(f) =P_n\bigparen{\sum_{n,k\in\Z}\ip{f}{h_{n,k}}M_{a^{-1}k}T_{an}h + \ip{f}{h_{\alpha}}M_\alpha h}= \sum_{k\in\Z}\ip{f}{h_{n,k}}M_{a^{-1}k}T_{an}h,
\end{equation}
where $P_n$ denotes the orthogonal projection onto $L^2([na,(n+1)a].$
By the uniqueness of the alternative dual of $\set{M_{a^{-1}k}T_{an}h}_{n,k\in\Z}$, it follows that 
\begin{equation}
\label{alternative_dual_form1}
    h_{n,k}\Eq M_{a^{-1}k}T_{an}h \Eq T_{an}M_{a^{-1}k}h^{-1} \quad \text{a.e. for any } (n,k)\in (\Z\setminus\set{0})\times \Z.
\end{equation}
For the case $n=0$, since $\alpha\notin a^{-1}\Z$, there exists an infinite sequence of nonzero scalars $(c_k)_{k\in\Z}\in \ell^2(\Z)$ such that $M_{\alpha}h=\sum_{k\in\Z}c_kM_{a^{-1}k}h$. Then for $f\in L^2(\R)$ that is supported on $[0,a]$ we have 
\begin{align}
\label{alternative_dual_form2}
\begin{split}
    f\Eq \sum_{k\in\Z}\ip{f}{h_{0,k}}M_{a^{-1}k}h + \ip{f}{h_{\alpha}}M_\alpha h&\Eq \sum_{k\in\Z}\ip{f}{h_{0,k}}M_{a^{-1}k}h+\sum_{k\in\Z}\ip{f}{h_\alpha}c_kM_{a^{-1}k}h\\
    &\Eq \sum_{k\in \Z} \ip{f}{h_{0,k}+\overline{c_k}h_\alpha}M_{a^{-1}k}h.
    \end{split}
\end{align}
Likewise, we have $h_{0,k}+ \overline{c_k}h_\alpha  = M_{a^{-1}k}h^{-1}$, and hence
$$h_{0,k} = M_{a^{-1}k}h^{-1} - \overline{c_k}h_\alpha \quad \text{a.e. for all }k\in\Z.$$

Now, suppose to the contrary that there exist some $\phi\in L^2(\R)$ and some countable subset $\Gamma=\set{(\beta_{n,k},\gamma_{n,k})}\cup \set{(\beta_{\alpha},\gamma_{\alpha})}$ in $\R^2$ such that 
\begin{equation}
\label{set_equal}
h_{\alpha}\Eq T_{\beta_{\alpha}}M_{\gamma_{\alpha}}\phi\quad \text{and}\quad h_{n,k}\Eq T_{\beta_{n,k}}M_{\gamma_{n,k}}\phi\quad \text{for any }n,k\in\Z.
\end{equation}
Consequently, we have $|h_{n,k}|=|T_{\beta_{n,k}}\phi|$ for any $n,k\in\Z.$ Therefore, the diameters of the supports of $h_{n,k}$ are equal for all $n,k\in\Z$. Moreover, for any $n,k,n_1,k_1\in\Z$, one can always obtain $|h_{n,k}|$ by applying certain translation to $|h_{n_1,k_1}|$. By Equation (\ref{alternative_dual_form1}), $h_{n,k}$ is supported on $[na,(n+1)a]$ for any $(n,k)\in (\Z\setminus\set{0})\times \Z$. It follows that for any $k\in\Z$ 
\begin{equation}
\label{magnitude_equation}
|h^{-1}|\Eq |T_{-an}h_{n,k}|=|h_{0,k}|=|M_{a^{-1}k}h^{-1} - \overline{c_k}h_\alpha| \quad\text{a.e.}
\end{equation}
By taking the square, we obtain that for any $k\in \Z$
\begin{equation}
|h^{-1}|^2=|h^{-1}|^2-2\text{Re}\bigparen{e^{2\pi ia^{-1}kx}h^{-1}c_k\overline{h_\alpha}}+|c_kh_\alpha|^2. 
\end{equation}
Equivalently, 
\begin{equation}
\label{magnitude_equation2}
2\text{Re}\bigparen{e^{2\pi ia^{-1}kx}h^{-1}\overline{h_\alpha}c_k}\Eq |c_kh_\alpha|^2.
\end{equation}
Let $\theta_1(x)$ and $\theta_2(x)$ be real-valued functions such that $h^{-1}=|h^{-1}|e^{i\theta_1}$ and $\overline{h_\alpha}=|h_\alpha|e^{i\theta_2}$. We then rewrite the left-hand side of Equation (\ref{magnitude_equation2}) as 

\begin{align}
\begin{split}
2\text{Re}\bigparen{e^{2\pi ia^{-1}kx}h^{-1}\overline{h_\alpha}c_k}&\Eq 2|h^{-1}||h_\alpha|\text{Re}(e^{i(2\pi a^{-1}kx+\theta_1+\theta_2)}c_k)\\
&\Eq 2|h^{-1}||h_\alpha|\bigparen{\text{Re}(c_k)\cos(\psi_k(x))-\text{Im}(c_k)\sin(\psi_k(x))},
\end{split}
\end{align}
where for each $k\in\Z$ we define $\psi_k(x)=2\pi a^{-1}kx+\theta_1(x)+\theta_2(x).$
Since $|h_\alpha(x)|$ is nonzero almost everywhere on $[0,a]$, by Equation (\ref{magnitude_equation2}) we further obtain 
\begin{equation}
\label{magnitude_equation3}
\cos(\psi_k(x))-\frac{\text{Im}(c_k)}{\text{Re}(c_k)}\sin(\psi_k(x))=\frac{|c_k|^2}{2\text{Re}(c_k)}|h_\alpha(x)||h(x)|\quad \text{a.e. on }[0,a]. 
\end{equation}
By Lemma \ref{fundamental_lemmaI}(b), the right-hand side of Equation (\ref{magnitude_equation3}) converges to $0$ as $k$ tends to $\infty$ for almost every $x\in [0,a]$. Thus, it remains to show that the left-hand side does not converge to $0$ on a set of positive measure of $[0,a].$ By Lemma \ref{fundamental_lemmaI}(a), $\frac{\text{Im}(c_k)}{\text{Re}(c_k)}=C$ for some nonzero constant $C.$ It then follows that \begin{align}
    \begin{split}
        \cos(\psi_k(x))-\frac{\text{Im}(c_k)}{\text{Re}(c_k)}\sin(\psi_k(x))&\Eq \cos(\psi_k(x))-C\sin(\psi_k(x))\\
        &\Eq \sqrt{1+C^2}\cos\bigparen{\psi_k(x)-\eta},
    \end{split}
\end{align}
where $\eta=\tan^{-1}(C).$ By Lemma \ref{fundamental_lemmaII}(b), we see that for almost every $x\in [0,a]$ the left-hand side of Equation (\ref{magnitude_equation3}) does not converge to $0$ as $k\rightarrow \infty.$\medskip

\emph{Case 2 Pseudo-synthesis Duals}$\colon$ Let $\set{h_{n,k}}_{n,k\in\Z}\cup \set{h_{\alpha}}\subseteq L^2(\R)$ be a  pseudo-synthesis dual of $\Gc(h,\Lambda')$. That is, 
\begin{equation}
\label{alternative_dual_recon2}
    f\Eq \sum_{n,k\in\Z}\ip{f}{M_{a^{-1}k}T_{an}h}h_{n,k} + \ip{h}{M_\alpha h}h_{\alpha} \quad \text{for all }f\in L^2(\R),
\end{equation}
with the convergence of the series in the norm of $L^2(\R).$ Fix $n\in\Z\setminus\set{0}$ and let $f\in L^2(\R)$ be any function that is supported on $[na,(n+1)a]$. Then we have 
\begin{equation}
    f\Eq P_n(f) =\bigparen{\sum_{n,k\in\Z}\ip{P_nf}{M_{a^{-1}k}T_{an}h}h_{n,k} + \ip{P_nf}{M_\alpha h}h_{\alpha}}= \sum_{k\in\Z}\ip{f}{M_{a^{-1}k}T_{an}h}h_{n,k}.
\end{equation}
By the uniqueness of the pseudo-synthesis dual of $\set{M_{a^{-1}k}T_{an}h}_{n,k\in\Z}$, it follows that 
\begin{equation*}
    h_{n,k}\Eq M_{a^{-1}k}T_{an}h \Eq T_{an}M_{a^{-1}k}h^{-1} \quad \text{a.e. for any } (n,k)\in (\Z\setminus\set{0})\times \Z.
\end{equation*}
The rest of the proof follows a similar argument to that in Case $1$.
\end{proof}

Finally, we close this paper with a few remarks.
\begin{remark} (a) The generator of the Gabor frame in Theorem \ref{main_thm} is not only smooth, but also extends to an entire function of exponential type at most $2\pi a.$ (\cite{You01})\smallskip

(b) One might suspect that a potential reason why a Gabor frame fails to possess an alternative dual or a pseudo-synthesis dual that is itself a Gabor system is closely related to the ``high density" of the associated set of time-frequency shifts. A classical notion that have been frequently used to quantify the density of a set in this context is the \emph{Beurling density}. The lower and upper Beurling density associated with a countable subset $\Lambda\subseteq\R^d$ are defined by 
\begin{equation}
\label{Beruling_density}
D^{-}(\Lambda)=\liminf_{r\rightarrow\infty}\inf_{x\in \R^d}\frac{\#\bigparen{\Lambda\cap B_r(x)}}{|B_r(x)|}\quad\text{and}\quad D^{+}(\Lambda)=\limsup_{r\rightarrow\infty}\sup_{x\in \R^d}\frac{\#\bigparen{\Lambda\cap B_r(x)}}{|B_r(x)|},
\end{equation}
respectively. If $D^{-}(\Lambda)=D^{+}(\Lambda)$, then we say $\Lambda$ has a uniform Beurling density (see \cite{Hei07} for more introductions to Beurling density). A straightforward computation shows that the associated set of time-frequency shifts in Theorem \ref{main_thm} has a uniform Beurling density $1.$ On the other hand, it is known that if $\Gc(g,\Lambda)$ forms a frame for $L^2(\R^d)$, then $D^{-}(\Lambda)\geq 1$ (\cite{CDH99}).
Thus, the set of time-frequency shifts associated with the Gabor frame constructed in Theorem \ref{main_thm} already achieves the minimal possible uniform Beurling density for a Gabor frame in this sense. \smallskip

(c) The smoothness of the generator is not the critical issue that prevent a Gabor frame from admitting alternative duals or pseudo-synthesis duals that are Gabor systems. Let $\Gc(h,\Lambda')$ be the same Gabor frame specified in the proof of Theorem \ref{main_thm}. Assume that there exists some alternative dual or pseudo-synthesis dual that can be written in the form  $\set{T_\beta M_{\gamma}\phi}_{(\beta,\gamma)\in \Gamma}$ for some $\phi\in L^2(\R)$ and some subset $\Gamma\subseteq \R^2$. With slight modifications to the proof of Case 1 of Theorem \ref{main_thm}, one can show that there exists some $f\in L^2(\R)$ that is compactly supported and still satisfies Theorem \ref{main_thm} (a)--(c). 
\end{remark}

\end{document}